\documentclass{article}

\usepackage{amssymb,amsmath}

\newtheorem{theorem}{Theorem}
\newtheorem{proof}{Proof}

\newcommand{\const}{\hbox{\rm const}}

\def\R{\mathbb{R}}
\def\bC {{\hbox{\bf C}}}
\def\bR {{\hbox{\bf R}}}
\def\f {{\hbox{\bf f}}}
\def\bx {{\hbox{\bf x}}}
\def\bG {{\hbox{\bf G}}}
\def\bY {{\hbox{\bf Y}}}
\def\bX {{\hbox{\bf X}}}
\def\bZ {{\hbox{\bf Z}}}

\title
{\bf Justification of an asymptotic expansion at infinity}
\author
{\bf L.A. Kalyakin
\\
Institute of Mathematics,  Ufa Sci. Centre,  Russian Acad. of Sci.\\
112,Chernyshevsky str., Ufa,
\\450000, Russia\\
E-mail: klenru@mail.ru
\thanks
{The research was supported by RFBR grants nos. 06-01-00124, 06-01-92052,
INTAS grant no. 03-51-4286 and grant of Scientific School no. 2215.2008.1.}
\date{June. 02. 2008}}

\begin{document}

\maketitle

\begin{abstract}
\noindent A family of asymptotic solutions at infinity for the system of
ordinary differential equations is considered. Existence of exact solutions
which have these asymptotics is proved.
\end{abstract}

\section{Introduction}

This paper is a reply on the Calogero's question, asked on the NEEDS-2007,
about justification of an asymptotic solution for the ordinary differential
equations. The full-dimensional manifold of solutions is the principal
interest. Such solutions are usually linked with the physically observed
phenomena. In the case of linear equations an asymptotic expansion of general
solution can be constructed by using a fundamental system of asymptotic
solutions \cite{{Fedoruk_AsODU},{Olver}}. We present here a similar result
for nonlinear equations.

Main object  is the system of differential equations
\begin{equation}\begin{array}{lc}\displaystyle{\frac{d\bx}{dt}=t^k\f(\bx,t),\ \
\bx\in\R^n,\ (k=\const\geq 0)}\end{array}\label{Main_Eq}
\end{equation} considered at $t\geq t_0>0$. The factor $t^k$ is here involved
for identification of growing right sides, occurring in applied problems;
$k+1$ is called Poincare rank. Let $\bX(t;\alpha),\ \alpha\in\R^n$ be an
$n$--parametric asymptotic solution at infinity. It means that the vector
function $\bX(t)\in C^1[t_0,\infty)$ under substitution into the equation
\eqref{Main_Eq} yields the residual
\begin{equation}\begin{array}{lc}\displaystyle{\bY(t)\equiv\frac{d\bX}{dt}-
t^k\f(\bX,t),}\end{array}\label{FAR_Main_Eq}
\end{equation} which is vanishing at infinity:
$\bY(t)=O(t^{-\mu}),\ t\to\infty,\ \mu=\const>0$. We study the following
problem: is there $n$--parametric exact solution $\bx(t;\alpha)$, for which
the $\bX(t;\alpha)$ is an asymptotic approximation at infinity:
$\bx(t;\alpha)=\bX(t;\alpha)+O(t^{-\nu}),\ t\to\infty,\ (\nu=\const>0)$? An
affirmative answer is obtained here. The proof is reduced to the  existence
theorem for the remainder. The principal feature is the uniformity of the
asymptotics with respect to parameter $\alpha$.

Justification of asymptotics in general case was not proved up to now and,
probably, can not be proved at all, although very interesting results are
known for some types of equations. Results for the linear equations are
collected in \cite{{Fedoruk_AsODU},{Olver}}. Justification of asymptotics for
the nonlinear Painlev\'e equations are discussed in
\cite{{Butru},{Nov},{Nalini}}. Usually existence theorem follows the formal
asymptotic construction. Various theorems were proved under different
conditions which depend on the used asymptotic methods
\cite{{Bruno},{Kudr},{Kozl},{Kuzn1},{Kuzn2}}. The well known Kuznetsov's
results \cite{{Kuzn1},{Kuzn2}} deal with the case when $X(t)$ is a formal
series (of powers, logarithms and exponentials) with constant coefficients.
Sometimes such series give an asymptotic expansion of general solution
\cite{Kozl}. But often such expansion does not correspond to general
solution. For example, the pendulum equation $\ddot x+\sin x=0$ has few
solutions of that type: a single solution $x(t)\equiv 0$ and two
one-parametic solutions $x(t)=\pm\pi+O(\exp(-t)),\ t\to\infty$. The other
solutions are oscillating and belong to two-parametric family.

Usually the general asymptotic expansions for nonlinear equations are given
by WKB formulas in the form of power series with oscillating coefficients.
P.Boutroux studied such two-parametric asymptotics for the Painlev\'e
equations \cite{Butru}. In general case the WKB asymptotics depend on
$n$--parameters and just such solutions are of interest in physics. For
example, existence of two-parametric increasing solution of main resonance
equations is identified with existence of autoresonance phenomenon
\cite{{FajansFr},{LK_Stekl}}. Justification of WKB asymptotics for the linear
equations \cite{{Fedoruk_AsODU},{Wasov}} and for some nonlinear equations
\cite{{Kozl}} can be reduced to the Kuznetsov's theorem. But the reduction is
unknown and evidently is impossible for many equations occurred in physics.
Painleve\'s equations studied in \cite{{Butru},{Nov},{Nalini}} give examples
of such type.

Justification of WKB asymptotics for the Painlev\'e equations is usually
based on the property of integrability \cite{Nov}. Nonlinear nonintegrable
equations are studied in this paper. We prove an existence theorem and give
an estimate for the remainder. The presence of the $n$--parametric solution
makes very easy the proof. Both the technic of the proof and the derived
estimate are determined by the properties of the asymptotic solution
$\bX(t;\alpha)$. Similar estimate obtained in \cite{Ging_Tovbis} for the case
of single solution $\bX(t)$ is linked with the properties of matrix of
linearized system $\partial_\bx\f(\bX,t)$. Both of these results are obtained
under different {\it sufficient} conditions, hence in general case they can
not be reduced to each other. In order to perform a comparison one have to
specify the considered system and the asymptotic solution. We avoid both any
asymptotic construction and any discussion on the method which yields the
asymptotic solution. This approach allows us to obtain an existence theorem
in general form which may be useful for wide range of problems.

Our estimate is uniform with respect to parameter $\alpha$ and this feature
is crucial for the successful application in asymptotic theories. In
particular the uniformity of the WKB asymptotics proves the orbital stability
of the oscillating solutions. Researches \cite{{Kuzn1},{Kuzn2},{Ging_Tovbis}}
deal with a single solution which may be stable or unstable, see \cite{Kozl}.
As a rule the solution have no any physical sense, if it is unstable.
Stability of the solution is not discussed in the existence theorems
\cite{{Kuzn1},{Kuzn2},{Ging_Tovbis}}.  The problem of stability return us to
investigation of $n$--parametric solution which is considered here. Certainly
our result does not give any information on unstable solutions.

\section{ Input conditions}\label{Condition}

We consider the equation \eqref{Main_Eq} under assumption that the  vector
function from the right side $\f(\bx,t)$ is smooth (continuously
differentiable) in the domain $\{\bx\in D\subseteq\R^n, t\geq t_0\}$. An
asymptotics at infinity
\begin{equation*}\displaystyle{\f(\bx,t),\frac{\partial\f(\bx,t)}{\partial\bx}
=O(1), \quad t\to\infty}\end{equation*} uniformly on every compact
$\bx\in\mathcal{K}\subset D$ is supposed. The principal conditions are
formulated by means of the asymptotic solution $\bX(t;\alpha)$. It is assumed
that the vector function $\bX(t;\alpha)$ depends on the parameters
$\alpha=(\alpha_1,...,\alpha_n)$ and this dependence is smooth in some domain
$\alpha\in{A}_0\subseteq\R^n$. Values of the asymptotic solution belong to
the compact\footnote{The case of unbounded asymptotic solution can be reduced
to considered one.}: $\bX(t;\alpha)\in \mathcal{K}_\mathcal{A}\subset \R^n$
when both $t\in[t_0,\infty)$ and parameters $\alpha$ range over any compact:
$\alpha\in\mathcal{A}\subset A_0$. Parameters $\alpha$ are supposed
functionally independent and the Jacobian
\begin{equation*}\mathcal{J}(t;\alpha)
\equiv\displaystyle{\frac{\partial\bX}{\partial\alpha}}\end{equation*} has an
asymptotics
\begin{equation*}\displaystyle{\det\mathcal{J}(t;\alpha)=
t^p[a+o(1)],\ \ t\to\infty,\ (p=\const,\ a=\const\neq 0)}.\label{Yakobi}
\end{equation*} In some sense $\bX(t;\alpha)$ is a general
asymptotic solution.  Elements of the  Jacobian matrix $J=\{a_{i,j}\}$ have a
similar asymptotics (may be different order for various number $i,j=1,...,n$)
\begin{equation*}a_{i,j}(t;\alpha)=O(t^{q_{i,j}}),\ \ t\to\infty,\
(q_{i,j}=\const).\end{equation*} We shall use the more short formula which
may be considered as an asymptotic estimate:
\begin{equation}\mathcal{J}(t;\alpha)
= O(t^{q}),\ \ t\to\infty,\ (q=\max q_{i,j}).\label{M_Yakobi}
\end{equation}
Under these assumptions the inverse matrix has a like asymptotic estimate
\begin{equation}\mathcal{J}^{-1}(t;\alpha)=O(t^{r}),\ \ t\to\infty.
\label{Inv_M_Yakobi}
\end{equation}
One can see that the exponents satisfy the relations: $p\leq qn,\ r\leq
q(n-1)-p$.

The residual determined by the formula \eqref{FAR_Main_Eq} depends on the
parameters as well: $\bY=\bY(t;\alpha)$. We assume asymptotic estimates:
\begin{equation}\displaystyle{\bY(t;\alpha)=O(t^{-\mu}),\
\frac{\partial\bY}{\partial\alpha}}= O(t^{-\mu+s}),\ \ t\to\infty,\
(\mu,s=\const).\label{Nevyazka}
\end{equation}
All asymptotics given above are supposed to be uniform with respect to
parameters $\alpha\in\mathcal{A}$.

{\bf Remark.}  All restrictions are induced by examples of WKB asymptotics
for the Painleve's equations \cite{{Butru},{Nov},{Nalini}} and for the main
resonance equations \cite{{LK_Stekl}}. The exponent $\mu$ is a crucial
parameter. It corresponds to the length of the asymptotic solution. As a
rule, $\mu$ can be made large as desired\footnote{In the case of complicated
problem the quantity $\mu$ depends on the talent of researcher,
\cite{Babich}.}. While the exponents $k,p,q,r,s$ depend on both the
considered equations and the family of asymptotic solutions. They do not
depend on the length of asymptotics $\mu$, as is seen in different examples.

\section{Existence theorem}\label{Exist}

\begin{theorem}
Let the vector function $\bX(t;\alpha),\ \alpha\in A_0\subseteq\R^n$ be an
$n$--parametric asymptotic solution of the equation \eqref{Main_Eq} and the
properties \eqref{M_Yakobi}, \eqref{Inv_M_Yakobi} hold.  If  the residual
defined by \eqref{FAR_Main_Eq} is decreased \eqref{Nevyazka} fast
sufficiently at infinity:
\begin{equation} \mu>r+s+1 \ {\rm and} \ \mu>2(r+q+1)+k,\label{M_Cond}
\end{equation} then for arbitrary compact
$\mathcal{A}\subset A_0$ there is a neighborhood of infinity $t\in[T,\infty)$
where the equation \eqref{Main_Eq} has the $n$-parametric exact solution
$\bx(t;\alpha)$ which has an asymptotics
\begin{equation*}\bx(t;\alpha)=\bX(t;\alpha)+O(t^{-\nu}),\ \
t\to\infty,\ (\nu=\mu-r-q-1>0)\end{equation*} uniformly with
$\forall\,\alpha\in\mathcal{A}$.
\end{theorem}

\begin{proof}
The idea of the proof was suggested by M.Fedoruk \cite{Fed-WKB} and later was
applied in \cite{LK_Stekl}.

We seek an exact solution as the sum
$\bx(t;\alpha)=\bX(t;\alpha)+\bR(t;\alpha).$ The dependence on the parameters
$\alpha$ will be omitted for short. The equation for the desired remainder
$\bR$ reads
\begin{equation}\displaystyle{\frac{d\bR}{dt}=
t^k[\f(\bX+\bR,t)-\f(\bX,t)] -\bY(t).}\label{Res_R_Eq}
\end{equation} It can be rearranged by
identification of the linear part of the operator as follows
\begin{equation}\displaystyle{\frac{d\bR}{dt}=
t^k[\mathcal{M}_0(t)\bR+\bG(\bR,t)] -\bY(t).}\label{Res_Eq}
\end{equation}
Here the matrix $\mathcal{M}_0(t)$ is calculated by means of derivatives of
the known functions given above: \begin{equation*}\mathcal{M}_0=\displaystyle
{\frac{\partial\f(\bX,t)}{\partial\bx}}.\end{equation*} Vector function
\begin{equation*}\bG(\bR,t)=\f(\bX+\bR,t)-\f(\bX,t)-\partial_{\bx}\f(\bX,t)\bR
\end{equation*} has an asymptotics
$\bG(\bR,t)=O(|\bR|^2)$ as $\bR\to 0$. This asymptotics is uniform with
respect to $t\geq t_0$.

We desire to demonstrate existence of an exact solution $\bR(t;\alpha)$ which
is given in some infinite interval $T\leq t<\infty$ and which is rapidly
decreasing at infinity. As usually in such theorems, the differential
equation \eqref{Res_Eq} must be reduced to an integral equation in order to
apply the  method of successive approximations. However, a trivial reduction
consisting in integration with respect to $t$ does not lead to success
because the linear part of the integral operator is not contractive in that
case. The idea suggested by M.Fedoruk consists in inversion of an essential
part of the linear operator in the equation \eqref{Res_Eq} by using the given
asymptotic solution $\bX(t;\alpha)$. For this type of inversion  we use the
Jacobi matrix $\mathcal{J}(t;\alpha).$ It is a fundamental asymptotic
solution for the linearized equation. Indeed, derivation of the starting
identity \eqref{FAR_Main_Eq} with respect to parameters $\alpha$ yields
\begin{equation*}\displaystyle{\frac{d\mathcal{J}}{dt}-t^k\mathcal{M}_0(t)
\mathcal{J} = \partial_\alpha\bY.}\end{equation*} Here the left side
represents linearized operator which is applied at $J$, while the right side
is the residual which is rapidly vanishing at infinity:
$\partial_\alpha\bY(t;\alpha)=O(t^{s-\mu}),\ t\to\infty.$

The Jacobi matrix is used in the substitution $\bR=\mathcal{J}\bC$. Then the
equation \eqref{Res_Eq} is transformed at the equation for the vector
function $\bC(t;\alpha)$:
\begin{equation}\displaystyle{\frac{d\bC}{dt}=
\mathcal{M}(t)\bC+\mathcal{G}(\bC,t) -\mathcal{J}^{-1}\bY(t).}
\label{Res_C_Eq}
\end{equation} Main novelty now is the matrix $\mathcal{M}$ at the linear
part
\begin{equation*}\mathcal{M}(t)=-\mathcal{J}^{-1}\partial_\alpha\bY=
O(t^{r+s-\mu}),\end{equation*} which is rapidly decreasing at infinity if the
$\mu$ is sufficiently large. In particular, the norm of the matrix is
estimated in $\R^n$
\begin{equation*}|\mathcal{M}(t)|\leq M_0t^{r+s-\mu},\ \forall t\geq t_0,\
(M_0=\const).\end{equation*} Nonlinear part is represented by the vector
function
\begin{equation*}\mathcal{G}(\bC,t)= t^{k}\mathcal{J}^{-1}\bG(\mathcal{J}\bC,t).
\end{equation*} It has an asymptotics $\bG(\bR,t)=O(|\bR|^2)$ as $|\bR|\to 0$
and this asymptotics is differentiable with $\bR$ (uniformly with respect to
$t$). Then the inequality
\begin{equation*}|\bG(\bR_1,t)-\bG(\bR_2,t)|\leq M_K
(|\bR_1|+|\bR_2|)|\bR_1-\bR_2|, \ \ \forall\,t\geq t_0\end{equation*} take
place for all $\bR_1,\bR_2$ in arbitrary compact $\mathcal{K}$; here the
constant $M_K$ depends on the chosen compact $\mathcal{K}$. Last relation
provides the Lipschitzian property of the operator:
\begin{equation}\begin{array}{lc}|\mathcal{G}(\bC_1,t)-\mathcal{G}(\bC_2,t)|\leq
M_Kt^{k+r}\cdot t^q(|\bC_1|+|\bC_2|)\cdot t^q|\bC_1-\bC_2|\end{array}
\label{Lip}
\end{equation} for all vector function
$t^q\bC_1(t),t^q\bC_2(t)$ in every compact $\mathcal{K}$.

Differential equation \eqref{Res_C_Eq} is equivalent to the integral equation
\begin{equation}\bC(t)-\int_\infty^t[\mathcal{M}(\eta)\bC(\eta)
+\mathcal{G}(\bC(\eta),\eta)]\,d\eta=\bZ(t).\label{Res_Int_Eq}
\end{equation}
The vector function from the right side
\begin{equation*}\bZ(t)=\int_t^\infty\mathcal{J}^{-1}\bY(\eta)\,d\eta\end{equation*}
is well known and has an asymptotics $\bZ(t)=O(t^{r+1-\mu}),\ t\to\infty.$ We
are looking for a solution $\bC(t)$ which has a like asymptotic behavior.

In order to prove the theorem, we have to introduce some Banach space. An
appropriate space of continuous vector functions $\bC(t)\in
C_\lambda[T,\infty)$ is determined by the weight norm
\begin{equation*}\|\bC\|_\lambda=\sup_{t\geq T}t^{\lambda}|\bC(t)|,\ \lambda=
\mu-r-1.\end{equation*} The boundary $T=\const\geq t_0$ will be chosen from
requirement of contractiveness of the integral operator in what follows.

The first condition from \eqref{M_Cond} taken in the form $\lambda> q$ and
the inequality \eqref{Lip} provide the estimate
\begin{equation}\begin{array}{lc}|\mathcal{G}(\bC_1,t)-\mathcal{G}(\bC_2,t)|
\leq
M_Kt^{k+r+2q-2\lambda}\cdot(\|\bC_1\|_\lambda+\|\bC_2\|_\lambda)\|\bC_1-\bC_2\|_\lambda\end{array}
\label{Lip-2}
\end{equation} which is valid for all vector functions
$\bC_1(t),\bC_2(t)$ in arbitrary compact $\mathcal{K}\subset
C_\lambda[T,\infty)$. If we take here integral with $t$, then we can derive
the Lipschitzian inequality
\begin{equation}\begin{array}{lc}\displaystyle\Big\|\int_\infty^t
[\mathcal{G}(\bC_1,\eta)-\mathcal{G}(\bC_2,\eta)]\,d\eta\Big\| _\lambda\leq
M_K2KT^{k+r+2q+1-\lambda}\cdot\|\bC_1-\bC_2\|_\lambda \equiv
L_K\cdot\|\bC_1-\bC_2\|_\lambda,\end{array} \label{Lip-3}\end{equation} which
is valid for all $ \|\bC_1\|_\lambda,\|\bC_2\|_\lambda\leq K$. The
Lipschitzian constant
\begin{equation*}L_K=M_K2KT^{k+2(r+q+1)-\mu}\end{equation*} can be chosen
arbitrary small (for example, $L_K<1/2$) under large $T$, because the
exponent $k+2(r+q+1)-\mu<0$ is negative because of the second condition
\eqref{M_Cond}. The relation \eqref{Lip-3} provides contractiveness of the
nonlinear part of the integral operator.

The norm of the linear part of the integral operator can be estimated by
integration of the matrix norm $\mathcal{M}(t)$:
\begin{gather*}\sup_{t\geq T}t^{\lambda}\int_t^\infty\eta^{-\lambda}
|\mathcal{M}(\eta)|\,d\eta \leq M_1\sup_{t\geq T}t^{\lambda}\int_t^\infty
\eta^{r+s-\mu-\lambda}\,dt=\\ \\
=M_1T^{r+s+1-\mu}/(r+s-\mu-\lambda+1)\equiv L_0.\end{gather*} The constant
$L_0$ can be made small as desired (say, $L_0<1/2$) under assumed condition
$\mu>r+s+1$. It is enough to choose an appropriate large $T$. The $L_0$ is
used as a Lipschitzian constant for the linear part of integral operator:
\begin{equation*}\Big\|\int_\infty^t\mathcal{M}(\eta)[\bC_1(\eta)
-\bC_1(\eta)]\,d\eta\Big\|_\lambda \leq
L_0\|\bC_1-\bC_2\|_\lambda.\end{equation*}

Thus, if the value of $T$ is sufficiently large then the integral operator
from the equation \eqref{Res_Int_Eq} turn out to be contractive in Banach
space $C_\lambda[T,\infty)$. The right side of the equation
$\bZ(t)=O(t^{r+1-\mu}),\ t\to\infty$ belongs the same space $\bZ(t)\in
C_\lambda[T,\infty)$ as $\lambda=\mu-r-1$. Then the fixed-point theorem can
be applied to the integral equation \eqref{Res_Int_Eq}. It implies existence
of the solution $\bC(t)\in C_\lambda[T,\infty)$. This vector function is
differentiable because of the identity \eqref{Res_C_Eq}. As the equation
\eqref{Res_C_Eq} is equivalent to \eqref{Res_R_Eq}, then
$\bR(t)=\mathcal{J}\bC(t)$ is a desired solution of the residual equation
\eqref{Res_R_Eq} and it has an asymptotics
\begin{equation*}\bR(t)=\mathcal{J}\bC(t)=O(t^{q-\lambda}),\ t\to\infty,
\ (\lambda-q=\nu).
\end{equation*}

It must be point out that all found estimates are uniform with respect to
parameters  $\alpha$ in arbitrary compact $\mathcal{A}$. So the proof, as
given above, is suitable for the both case either fixed parameter $\alpha$ or
fixed compact: $\alpha\in\mathcal{A}\subset A_0$. In the last case the
estimate of the remainder
\begin{equation*}\bR(t;\alpha)=\mathcal{J}\bC(t;\alpha)=O(t^{-\nu}),
\ t\to\infty\end{equation*} is uniform with respect to
$\forall\,\alpha\in\mathcal{A}$. This completes the proof.

\end{proof}

{\bf Remark.} It may be occurred in application at specific problems that the
order of the asymptotic estimate of the remainder $\bR(t)=O(t^{\mu-r-q-1})$
is very rough. It is lack of the theorem given for general case. More precise
estimations can be derived, if a full asymptotic solution in the form of
infinite asymptotic series is known; see \cite{LK_Stekl}.

\end{document}